\documentclass[a4paper,11pt]{amsart}
 
\usepackage{amsmath,amsfonts,amsthm,amssymb,graphicx} 

\AtBeginDocument{%
  \newcommand\sslash{\mathbin{/\mkern-5.5mu/}}%
} 

\usepackage{enumerate,hyperref,fullpage}

\theoremstyle{plain}
\newtheorem{thm}{Theorem}[section]
\newtheorem{lemm}[thm]{Lemma}
\newtheorem{lemma}[thm]{Lemma}

\newtheorem{mainthm}{Theorem}
\newtheorem{maincoro}{Corollary}
\theoremstyle{definition}

\newtheorem{definition}[thm]{Definition}
\newtheorem*{acknowledgement}{Acknowledgement}
\theoremstyle{remark}
\newtheorem{rem}[thm]{Remark}

\newcommand{\bP}{\mathbb{P}}

\newcommand{\bN}{\mathbb{N}}
\newcommand{\bZ}{\mathbb{Z}}
\newcommand{\bG}{\mathbb{G}}
\newcommand{\bR}{\mathbb{R}}

\newcommand{\bS}{\mathbb{S}}

\newcommand{\cO}{\mathcal{O}}
\newcommand{\cL}{\mathcal{L}}

\DeclareMathOperator{\SL}{SL}
\DeclareMathOperator{\GL}{GL}
\DeclareMathOperator{\PGL}{PGL}

\DeclareMathOperator{\BT}{BT}
\DeclareMathOperator{\MIL}{MinInvLoc}
\DeclareMathOperator{\MRL}{MinResLoc}
\DeclareMathOperator{\Hom}{Hom}

\DeclareMathOperator{\Res}{Res}
\DeclareMathOperator{\Int}{Int}
\DeclareMathOperator{\diag}{diag}
\DeclareMathOperator{\id}{id}
\DeclareMathOperator{\ev}{eval}
\DeclareMathOperator{\Stab}{Stab} 
\DeclareMathOperator{\Spec}{Spec}

\DeclareMathOperator{\rank}{rank}

\DeclareMathOperator{\spe}{sp}
\DeclareMathOperator{\RSS}{SSRL}
\DeclareMathOperator{\RS}{SRL}
\DeclareMathOperator{\sst}{\mathrm{sst}}
\DeclareMathOperator{\st}{\mathrm{st}}
\DeclareMathOperator{\ssr}{\mathrm{sst-red}}
\DeclareMathOperator{\sr}{\mathrm{st-red}}

\DeclareMathOperator{\conv}{Conv} 
\DeclareMathOperator{\grad}{grad}

\newcommand{\sln}[1]{\operatorname{SL}_{#1}}

\newcommand{\ol}[1]{\overline{#1}}

\newcommand{\shf}[1]{\mathcal{#1}}
\newcommand{\lshf}{\mathcal{L}}
\newcommand{\oshf}{\mathcal{O}}

\newcommand{\zahl}{\mathbb{Z}}
\newcommand{\rat}{\mathbb{Q}}
\newcommand{\real}{\mathbb{R}}

\newcommand{\proj}{\mathbb{P}}
\newcommand{\Proj}{\operatorname{Proj}}

\newcommand{\gmult}{\mathbb{G}_{m}}
\newcommand{\oring}{\mathcal{O}}

\newcommand{\ord}{\operatorname{ord}}

\DeclareMathOperator{\chara}{char}

\newcommand{\invs}{I_m}
\newcommand{\invsgen}[1]{I_{#1}}
\newcommand{\obody}{B_{\Lambda , x}}
\newcommand{\bdobody}{\underline{\partial}B_{\Lambda ,x}}

\numberwithin{equation}{section}

\title[]{Semistable reductions and minimalities of invariants for group scheme actions on projective schemes}
\author{Rin Gotou}
\address{Graduate School of Mathematical Sciences, The University of Tokyo, 3-8-1 Komaba Meguro-ku Tokyo 153-8914, Japan}
\email{gotou-rin@g.ecc.u-tokyo.ac.jp}

\author{Y\^usuke Okuyama}
\address{Division of Mathematics, Kyoto Institute of Technology, Sakyo-ku, Kyoto 606-8585 JAPAN}
\email{okuyama@kit.ac.jp}

\date{\today}

\begin{document}

\begin{abstract}
Let $K$ be an algebraically closed and complete non-archimedean and non-trivially valued field,
and let $G$ be a reductive group scheme acting on a flat projective scheme $X$ defined over the base ring of $K$-integers. For every $K$-point $x$ in $X$, we introduce the minimal invariant locus $\MIL_x$ and the semistable reduction translation locus $\RSS_x$ 
in the translation space $\BT_G(K)$ associated with $G_K$, which is a variant of Bruhat-Tits building, and 
establish not only the coincidence of those loci
but, under a mild completeness assumption, also their non-emptiness. In the dynamical setting which has been studied by Szpiro--Tepper--Williams and Rumely, the coincidence result is already new in higher dimensions, and the non-emptiness result includes Rumely's $1$-dimensional result
at least in the spherical complete case.
\end{abstract}
\maketitle

\section{Introduction}
\label{sec:intro}

let $K$ be an algebraically closed and complete non-archimedean and non-trivially valued field, and 
\begin{gather*}
\ord:K\to\bR\sqcup\{+ \infty\}
\end{gather*}
the order function on $K$ induced by the absolute value of $K$. Let us denote by $\cO_K$ and $\mathfrak{m}_K$ the
ring of $K$-integers and its unique maximal ideal, respectively, and by $k=k_K$ the residual field of $K$ (and then $k=\overline{k}$).   
Let $X$ be a flat $\oring_K$-projective scheme and $\cL$ an ample line bundle (a.k.a.\ ample invertible sheaf) of $X$. 
For every $m\in\bN$, the $\cO_K$-module 
\begin{gather*}
 L_m:=H^0(X, \cL^m)
\end{gather*}
is an $\oring_K$-lattice in the $K$-vector space $V_m:= H^0(X_K , \cL^m_K )$ canonically.

Let also $G$ be a reductive $\oring_K$-group scheme (as in \cite[Exp.XIX]{SGA3vol3}, a reductive group scheme is always smooth so flat here) acting on $X$, and we assume that $\cL$ is 
$G$-linearized by a linearization sheaf isomorphism 
$\sigma$ (see Section \ref{sec:git} for the details).
Let us focus on the totality of $\sigma_*(G_K)$-invariant sheaf sections
\begin{gather*}
    \invs :=L_m \cap V_m^{G_K},\quad\text{where}\quad V_m^{G_K}=\{s\in V_m:g\cdot s=s\text{ for any }g\in G(K)\},
\end{gather*}
of $\cL^m$ for each $m\in\bN$, and let us call the left orbit space 
\[ 
\BT_G(K) :=  G(\cO_K) \backslash (G(K)) 
\] 
the translation space associated to $G_K$, where the notation comes from ``Bruhat-Tits''.

First, for $m\gg 1$ (so that $\cL^m$ is very ample), 
to every $x\in X_K(K)$ and every $\rho\in \invs$ such that $\rho(x)\neq 0$, we associate 
the {\em order function $\ord \rho_x$ on $\BT_G(K)$} 
so that for every $[g]\in\BT_G(K)$,
\begin{gather}
(\ord\rho_x)([g]) 
= \ord\bigl(\rho (x)\bigr) - \min_{s \in L_m(\subset V_m)}
\ord \bigl(( \sigma_*(g) \cdot s) (x)\bigr);\label{eq:invord}
\end{gather}
the {\em minimum locus} $\MIL_x$ in $\BT_G(K)$ of this function $\ord\rho_x$ is independent of $m \gg 1$, of $\rho\in \invs$, and even of the choice of evaluation isomorphisms $\lshf^m|_x \simeq K$ of $K$-vector spaces (see Lemma \ref{th:const} and Remark \ref{th:eval}), 
and is called the {\em minimal invariant locus} for $x$. 
Such loci have been studied in the setting of the conjugation actions of $G=\PGL_{n+1}$, $\GL_{n+1}$, and $\SL_{n+1}$ on $X=X_{n,d}=\bP^{N-1}\supset\Hom_d(\bP^n)$,
$n,d\in\bN$,
equipped with $\cL = \oshf_{\proj^{N-1}}(1)$, where
$N=N_{n,d}$ is written down in \eqref{eq:dim} below, and the invariant sheaf section $\rho$
treated there was specifically the Macaulay resultant form $\Res$ on $\bP^{N-1}$; so, in this dynamical setting, our $\MIL_\varphi$, $\varphi\in(\Hom_d(\bP^n))(K)$, has been studied as the minimal {\em resultant} locus $\MRL_\varphi$ (Silverman \cite{silverman2007arithmetic}, Szpiro--Tepper--Williams \cite{SzpiroTepperWilliams2014simistabminres}, and Rumely \cite{Rumely17}; see also \cite{Okugeometric}).

Next, let $X^{\sst (\text{resp.\ }\st)}(\cL)$ be the semistable (resp.\ stable) locus in $X_K(K)$ under the $G_K$-action, and similarly, let $X^{\sst (\text{resp.\ }\st)}_k(\cL_k)$ be the semistable (resp.\ stable) locus in $X_k$ under the $G_k$-action; see Section \ref{sec:git} for the details on GIT-(semi)stability. The $G(\cO_K)$-invariant subset 
\begin{gather*}
 X^{\ssr (\text{resp.\ }\sr)}(\cL):= \spe^{-1}\bigl(X_k^{\sst (\text{resp.\ }\st)}(\cL_k)\bigr)
\end{gather*}
in $X_K(K)$ is called the semistable (resp.\ stable) reduction locus in $X_K(K)$, where 
\begin{gather*}
    \spe:X_K(K)\to X_k(k) 
\end{gather*}
is the specialization (or mod $\mathfrak{m}_K$ reduction) map.
The $\ssr$ (resp.\ $\sr$)-translation locus in $\BT_G (K)$ for a point $x\in X_K(K)$ is defined as
\begin{align*}
    \RSS_x \ (\text{resp.\ }\RS_x ) 
:= \bigl\{ [g]  \in \BT_G (K) : g \cdot x \in X^{\ssr (\text{resp.\ }\sr)}(\cL) \bigr\},
\end{align*}
which is certainly independent of $\rho\in \invs$ (and of $m\in\bN$); in the above mentioned studies of the $\PGL_{n+1}$-, $\GL_{n+1}$-, and $\SL_{n+1}$-conjugation actions on $\bP^{N-1}$, the relationship between $\MRL_\varphi$ and $\RSS_\varphi$ has been investigated, and when $n=1$, so does even their coincidence.

\subsection{Main results} The following is one of our principal results, where we establish such a coincidence between $\MIL_x$ and $\RSS_x$ in our more general setting.

\begin{mainthm} \label{th:genmininvloc} 
Let $K$ be an algebraically closed and complete non-archimedean and non-trivially valued field, 
let $X$ be a flat $\oring_K$-projective scheme, 
$G$ a reductive $\oring_K$-group scheme acting on $X$, 
and $\shf{L}$ an $G$-linearized ample line bundle of $X$, and assume that the base change $G_k$ of $G$ over the residual field $k$ is still reductive.
Then for every $x\in X_K(K)$, we have 
\begin{gather}
 \MIL_x=\RSS_x,
\end{gather}
and moreover, if $\RS_x \neq \emptyset$, then we have $\RS_x = \RSS_x$ and $\RS_x$ is a singleton. 
\end{mainthm}

Theorem \ref{th:genmininvloc} yields the following, which generalizes Rumely's result in dimension $n=1$, where $\rho=\Res$, in the dynamical setting (see also \cite{SzpiroTepperWilliams2014simistabminres} for one inclusion
$\MRL_\varphi\supset\RSS_\varphi$ in higher dimension $n$).
\begin{maincoro}
Let $K$ be an algebraically closed and complete non-archimedean and non-trivially valued field, and pick any $n,d\in\bN$. Then under the $G=\PGL_{n+1}$, $\GL_{n+1}$, and $\SL_{n+1}$-conjugation actions on $\bP^{N_{n,d}-1}$ equipped with the $G$-linearized line bundle $\cO_{\bP^{N_{n,d}-1}}(1)$, for every 
$\varphi\in X_{n,d}$, we have $\MIL_\varphi=\RSS_\varphi$.

In particular, for every $\varphi \in \Hom_d(\bP^n)$, $\MRL_\varphi=\RSS_\varphi$.
\end{maincoro}

The following non-emptiness of $\MIL_x$, so of $\RSS_x$, is also one of our principal results. 
For technical reasons, 
we introduce some $\real$-building $\BT_G^{\real}(K)$ densely including $\BT_G(K)$ 
(see Subsection \ref{sec:building} for the details), which is also used in the proof of Theorem \ref{th:genmininvloc}.

\begin{mainthm}\label{th:existence}
In the setting of Theorem \ref{th:genmininvloc},
suppose in addition that
the $\bR$-building $\BT_G^{\real}(K)$ is metrically complete. Then for any $x\in X^{\st}(\cL)$, we have 
\begin{gather}
    \MIL_x\neq \emptyset. 
\end{gather}
\end{mainthm}
The completeness of $\BT_G^{\real}(K)$ might also have its own interest.
In \cite{MSSS2013MetCompBTBldg}, this has been obtained when $K$ is spherically complete and a kind of boundary structure of $\BT_G(K)$,
the Moufang polygon, falls into some types. For example, in the dynamical setting, in $n=1$ (and $G=\PGL_2$, $\GL_2$, or $\SL_2$) case, $\BT_G^{\real}(K)$ is complete (remarkably, in this $n=1$ dynamical case, the spherical completeness assumption is known to be redundant \cite{Rumely17}).
We postpone in treating more closely this completeness matter to our forthcoming paper.

For the related works to Theorems \ref{th:genmininvloc} and \ref{th:existence} in the setting of discrete valued or universally Japanese rings, see \cite{Levy2012SemistableReduction} and \cite[Chapters 3 and 4]{Maculan2017DiophantineAppliOfGIT} (after \cite{Burnol92}); at least in such cases, the completeness assumption of $\BT_G^{\real}(K)$ in Theorem \ref{th:existence} could be omitted.

\subsection{Organization of the paper}
In Section \ref{sec:background}, we prepare some details on GIT and buildings, and in Section \ref{sec:wpandordfcn}, we study some properties of the order functions on $\BT_G(K)$. In Sections \ref{sec:coincide} and \ref{sec:SphComp}, we show Theorems \ref{th:genmininvloc}
and \ref{th:existence}, respectively. In Section \ref{sec:apply}, we give application of Theorem \ref{th:genmininvloc} and computation of $\MIL_\varphi$, in the dynamical setting in higher dimensions.

\section{Background on GIT and buildings}\label{sec:background}
We denote by $e=e_G$ the unit in a group $G$,
$\conv$ means the convex hull, and $\Int S$ means the interior of a subset $S$ in $\bR^q$. 

Let us fix a base ring $R$ which is a subring of a field $F$. Let $G=(G,\mu)$ be an algebraic group (over $R$) and
$\alpha : G \times X \to X$ the algebraic action of $G$
on an algebraic scheme $X$ (that is, $X$ is a $G$-scheme over $R$). 

\subsection{Geometric invariant theory}\label{sec:git}
Let $\cL$ be a line bundle on $X$ which is $G$-linearized by a (linearization) sheaf isomorphism 
\begin{gather*}
 \sigma:\alpha^* \cL\to \bigl(\text{the projection }p_2:G \times X\to X\bigr)^*\cL 
\end{gather*}
in that the cocycle condition 
\begin{gather*}
 (\mu \times \id_X )^* \sigma = \bigl((\text{the projection }p_{23}:G\times (G\times X)\to(G\times X))^* \sigma \bigr) \circ \bigl((\id_G \times \alpha )^* \sigma\bigr)
\end{gather*}
holds on $G\times G\times X$. 
Then by duality, there is a morphism 
\begin{gather*}
 \sigma^*:H^0(X , \cL ) \xrightarrow{\alpha^*} H^0(G \times X , \alpha^* \cL ) \xrightarrow{\sigma} H^0(G\times X, p_2^* \cL) \simeq H^0(G , \cO_G) \otimes H^0(X, \cL)
\end{gather*}
of $R$-modules, which in turn induces a representation
\begin{gather*}
 \sigma_*:G\ni g\mapsto\Bigl(\sigma_*(g) :  H^0(X,\cL ) \xrightarrow{\sigma^*} H^0(G,\oring_G) \otimes H^0(X,\cL ) 
\xrightarrow{\ev_g \otimes \id } F\otimes_R H^0(X,\cL)\Bigr)
\end{gather*}
on $F\otimes_RH^0(X,\cL)$ of $G$. 
We also note that for every $m\in\bN$, 
the $m$th power $\cL^m$ of $\cL$ is still
$G$-linearized by a
linearization sheaf isomorphism 
canonically induced by $\sigma$.

Suppose here that $R=F$, $G$ is reductive, 
$X$ is projective, and $\cL$ is ample. Then
we say a geometric point $x \in X$ is \emph{semistable} under $G$ if,
replacing $\cL$ with $\cL^m$ for some $m\in\bN$ if necessary, there is a $\sigma_{*}(G)$-invariant 
section $s \in H^0(X,\cL)$ which does not vanish at $x$
(\cite[Definition 1.7 and Amplification 1.11]{git1994});
in our situation, no affineness of the zero locus of $s$
is needed. The {\em semistable locus $X^{\sst}(\cL)$ under $G$} is a (Zariski) open subscheme of $X$ such that
the equality
\begin{gather*}
 \bigl(X^{\sst}(\cL)\bigr)_{\overline{F}}\bigl(\ol{F}\bigr) 
= \bigl\{ z \in X\bigl(\ol{F}\bigr): z \text{ is semistable under } G\bigr\}
\end{gather*}
holds. We also say a geometric point $x \in X^{\sst}(\cL)$ is \emph{stable} under $G$ if 
\begin{itemize}
 \item  the $G$-orbit 
 $G \cdot x:=\alpha(G,x)$ of $x$ is (Zariski) closed in the above
 $(X^{\sst}(\cL))_{\ol{F}}(\ol{F})$ and 
 \item the dimension of the stabilizer
 $\Stab(G,y)$ in $G$ of a geometric point $y\in X^{\sst}(\cL)$ near $x$
 is locally minimal at $x$. 
\end{itemize}
The {\em stable locus $X^{\st}(\cL)$ under $G$} is a (Zariski) open subscheme of $X^{\sst}(\cL)$ such that the equality 
\begin{gather*}
    \bigl(X^{\st}(\cL)\bigr)_{\ol{F}}\bigl(\ol{F}\bigr) = \bigl\{z\in X\bigl(\ol{F}\bigr):z\text{ is stable under }G\bigr\}
\end{gather*}
holds.

\begin{rem}\label{rem:actionlift}
The natural action of $G$ on the affine cone $\widetilde{X} \subset H^0(X,\lshf)^\vee$ of $X$ is that induced by the dual representation $\sigma_*^\vee$. In terms of  affine cones, a point $x \in X^{\sst}(\cL)$ is stable under $G$ if and only if $\sigma_*^\vee (G) \cdot \widetilde{x}$ is closed in $\widetilde{X}$ for any lifting $\widetilde{x}$ 
 to $\widetilde{X}$ of $x$.
\end{rem}

The categorical quotient $X^{\sst}(\cL )\sslash G$ exists, and
if there is no surjective group morphism $G \to \gmult$, then 
moreover $X^{\sst}(\cL )\sslash G  \simeq \Proj\bigoplus_{m\in\bN\cup\{0\}}H^0(X, \cL^m)$. In particular, all the $X^{\sst}(\cL)$, $X^{\st}(\cL)$, and $X^{\sst}(\cL )\sslash G$ are determined up to taking a finite power of $\cL$ (\cite[between Amplification 1.11 and Converse 1.12]{git1994});
this kind of uniqueness is not the case without any kind of assumption on the non-existence of a surjective group morphism $G \to \gmult$. 

\subsection{Weight polytope and Hilbert-Mumford criterion}\label{sec:wp}

Suppose here that $X$ is flat and projective and that $G$ is reductive. Let $\cL=(\cL,\sigma)$ be a $G$-linearized very ample line bundle on $X$ and set 
\begin{gather*}
 N_\lshf := \rank_R \bigl(H^0(X,\lshf )\bigr).
\end{gather*}
For every group scheme morphism $\Lambda : \bG_m^q  \to G$, $q \in \bN$, a section $s \in H^0(X,\lshf )$ is said to be of (pure) weight $w = (w_1,\ldots , w_q) \in \zahl^q$ if 
\[ 
\Lambda^*\bigl(  \sigma^* (s) \bigr)
 =\biggl(\prod_{j=1}^q t_j^{w_{j}}\biggr) s,
\quad \bG_m^q \simeq \Spec(R [t_1,\ldots , t_q]); 
\]
there is a basis $(b_i=b_i^{(\lshf ,\Lambda)})_{i=1}^{N_{\lshf}}$ of $H^0(X,\lshf )$ such that for each $i$, the section $b_i$ is of (pure) weight $w_i=(w_{i1},\ldots,w_{iq})\in\bZ^q$, 
and then for any base change to an $R$-algebra $S$ we have 
\begin{gather}
 \bigl(\sigma_*(\Lambda(t))\bigr)\cdot b_i 
 =\biggl(\prod_{j=1}^q t_j^{w_{ij}}\biggr) b_i,
\quad t=(t_1,\ldots, t_q )\in\bG_m^q(S) (\simeq (S^\times )^{q}  ), \label{eq:weitdecomp}
\end{gather}
and call such a basis $(b_i)$ a weight decomposition of $H^0(X,\cL)$ having the weights $(w_i)$.
To every point $x\in X$, the {\em weight polytope} 
\begin{align}
\Delta (x,\Lambda) :=& \conv \left( \{ w\in\bZ^q : \text{for some } s \in H^0(X,\lshf )
\text{ of weight }w, s(x) \neq 0 \} \right) \notag \\
=& \conv \bigl( \{ w_i \in \bZ^{q} : b_i (x) \neq 0 \}\bigr) \subset \real^q\label{eq:wpdecomp}
\end{align}
is associated. For more details, see \cite[Exp.VIII]{SGA3vol3}. 

It is remarkable that in seeing whether $x\in X$ is (semi)stable under $G$ or not, we do not need to know concretely what those $\sigma_*(G)$-invariant sections $s\in H^0(X,\cL)$ are. This remarkable fact is formulated as follows in terms of weight polytopes. For more details, see, e.g., \cite[Chapter 9]{LecInvTheory2003Dolgachev}.

\begin{thm}[Hilbert-Mumford criterion {\cite{git1994}}]\label{th:hmcrit}
 Let $G$ be an algebraic group, $X$ a projective $G$-scheme, and $\cL = (\cL,\sigma )$ 
a $G$-linearized very ample line bundle on $X$, over a base field $F$. Then for any geometric point $x \in X$, the following are equivalent;
\begin{enumerate}[(a)]
    \item \label{item:ss} $x\in X^{\sst}(\cL)$ (resp.\ $x\in X^{\st}(\cL)$).
    \item \label{item:nonnull}
For any injective group scheme morphism $\Lambda : \bG_m^q \to G$, we have $0 \in \Delta (x, \Lambda )$ (resp. $0 \in \Int\Delta (x, \Lambda )$).
\end{enumerate}
\end{thm}

\subsection{Maximal tori, lifting, and
the building associated to the translation space}\label{sec:building}

To fix notations, we recall the definition of an Euclidean building. The notion of convexity in affine spaces is defined as usual.

\begin{definition}[{\cite{WeissStrAffBldgs2009}, \cite{Struyve2011CpltRbldgAndFixptThm}}]\label{def:rbuild} 
Let $V$ be an $\real$-linear space and $A_0$ a $V$-affine space (i.e. a set with free transitive $(V, +)$-action), and set $O(A_0) := O(V) \ltimes (V,+)$ (i.e., the orthogonal affine transformation group on $A_0$). 

A subgroup $W$ of $O(A_0)$ is {\em an affine Weyl group
on $A_0$} if $W = \ol{W} \ltimes T_W$ for some finite subgroup $\ol{W}$ of $O(V)$ and 
some subgroup $T_W$ in $(V,+)$. A Weyl chamber of $\ol{W}$ is the closure of a connected component of 
the complement in $V$ of
\begin{multline*}
 \bigl\{ v\in V :  v \in V^{\sigma}\text{ for some }\sigma \in \ol{W} \text{ such that }\sigma^2 = e\text{ and } \dim (V^{\sigma}) = \dim V -1 \bigr\},\\
\text{where }V^\sigma = \{ v\in V : \sigma (v) = v \}.
\end{multline*}
A Weyl chamber on $A_0$ is any 
$(V,+)$-translation of a Weyl chamber of $\ol{W}$.

For an affine Weyl group $W$ on $A_0$, \emph{a Euclidean Building of type }$(A_0,W)$ is a tuple
\begin{gather*}
 \bigl(\text{a metric space }(B,d), \mathcal{A} = \{ f_i : A_0 \to B\}_{i \in I},
\{ r_{x,i} : B \to f_i(A_0) \}_{i \in I , x \in A_0}\bigr) 
\end{gather*}
satisfying the following axioms:
\begin{enumerate}[(i)]
    \item For any $i\in I$, $f_i: A_0 \to B$ is injective; both $f_{i}$ and the image $f_i(A_0)$ are called an apartment of $B$.
    \item\label{item:join} For any $x,y \in B$, there exists $i \in I$ such that $\{x,y\}\subset f_{i}(A_0)$.
    \item For any $f \in \mathcal{A}$ and any $w \in W$, $f \circ w \in \mathcal{A}$.
    \item \label{apcommon} For any $f,f' \in \mathcal{A}$, the set $C_{f,f'} := f^{-1}(f(A_0) \cap f'(A_0))$ is convex in the $V$-affine space $A_0$ and we have
$( (f')^{-1} \circ f) |_{C_{f,f'}} = w |_{C_{f,f'}}$ for some $w \in W$.
    \item For any $f \in \mathcal{A}$ and any Weyl chamber $W$ on $A_0$, there exists $f' \in \mathcal{A}$ such that $W \subset C_{f,f'}$.
    \item Each map $r_{x,i} : B \to f_i(A_0)$ is the nearest point retraction, which is 1-Lipschitz with respect to $d$. 
\end{enumerate}
\end{definition}

Shall we come back to our situation in Section \ref{sec:intro} (so $R=\cO_K$ and $F=K$). 
Since $\overline{K}=K$ (so $\overline{k}=k$),
$\cO_K$ is the only \'{e}tale extension of $\oring_K$, 
so we introduce herewith tori of $G$, $G_K$ and $G_k$
of only split type as follows. 
Recall that for any reductive group $H$ over a base ring $R$, a (split) {\em torus} of $H$ is an injective morphism $\Lambda : \bG_m^q \to H$ of $R$-group schemes for some integer $q \geq 1$, and a (split) {\em maximal torus} is a torus whose image is maximal among the images of tori. In our situation, for $H = G$, $G_K$, or $G_k$, the dimension of a maximal torus $T$ of $H$ is independent of the choice of $T$, and is called the rank $r(H)$ of $H$. Moreover, by \cite[Exp.\ XIX, Corollaire 2.6]{SGA3vol3}, the ranks of $G$, $G_K$, and $G_k$ are identical, and we set
\begin{gather*}
    r:=r(G)=r(G_K)=r(G_k)
\end{gather*}
here and below.

The following lifting property of maximal tori of $G$, which we state as a lemma, is crucial; see \cite[Exp.\ XII, Th\'{e}or\`{e}me 1.7]{SGA3vol2} for the details.

\begin{lemma}\label{lemm:toruslift}
Any maximal torus 
$\ol{\Lambda }: (\bG_m^r)_k  \to G_k$ lifts to a maximal torus $\Lambda : \bG_m^r \to G$.
\end{lemma}

Here and below,
\begin{gather}
    \Gamma := \ord (K^\times )\overset{\text{dense}}{\subset} \bR\label{eq:valgr}
\end{gather} 
is the valuation group of $K$. Using the $\bR$-linear space $\Gamma^r\subset\bR^r$, 
we can augment $\BT_G(K)$ by the following Euclidean building.
The Weyl group of the reductive group scheme $G_K$ over $K$ is the normalizer quotient 
$W_{G_K} = (N_{G_K}(T))/T$,
which is a finite subgroup of $O(\Gamma^r)$ and where $T$ is any image of maximal tori of $G_K$, which are mutually conjugate. 

\begin{thm}[{\cite{Tits1986Immubles}, see also \cite{WeissStrAffBldgs2009, Bacak2014ConvexAnalyAndOptimInHadamSp, RemyThuillierAmaury2015BTBldsAndAnaGeom}}]\label{rbuild}
Let $G$ be a reductive $\oring_K$-group. Then there exists a Euclidean building $\BT^\real_G(K)$ of type $(A_0,W)=(\real^r, W_{G_K} \ltimes \Gamma^r)$ such that
\begin{enumerate}[(i)] 
 \item \label{item:dense}
there is a $G(K)$-equivariant dense embedding
$\BT_G(K)\hookrightarrow \BT^\real_G(K)$.

\item \label{aparttor} There is a maximal torus $\Lambda : \gmult^{r} \to G$ such that for every apartment $f : \real^r \to \BT^\real_G(K)$, there is some $g \in G(K)$ such that for every $t=(t_1,\ldots,t_r)\in(K^\times)^r$, setting $v_t:=(\ord t_1,\ldots,\ord t_r)\in\Gamma^r$, 
\[ 
f(v_t) = [ \Lambda (t) \cdot g] \in \BT_G(K),
\]
so in particular that $f^{-1}(\BT_G(K)) = \Gamma^r$. Then we also denote $f$ by $f_{\Lambda,g}$. 
    \item The $\bR$-building $\BT^\real_G(K)$ is a metric $\operatorname{CAT}(0)$-space.
\end{enumerate}
\end{thm}

\section{Weight polytopes and the order functions}
\label{sec:wpandordfcn}
We denote by $\langle \cdot , \cdot \rangle$
the standard inner products on the Euclidean spaces $\real^q$, and for every $S\subset\bR^q$, set $-S:=\{-v\in\bR^q:q\in S\}$. For a function $c:X \to\bR\cup\{+\infty\}$ on a set $X$ and for any $u\in\bR$, set the sublevel set 
\begin{gather*}
 L_{c,\leq u} :=c^{-1}\bigl( (-\infty , u] \bigr)
\end{gather*}
here and below.

\subsection{Convex function}
To fix notations, we recall some facts 
on convex functions. 

For a continuous function $c : \real^q \to (-\infty,+\infty ]$, the {\em (convex) conjugate} $c^* :\bR^q\to\bR$ is the convex function defined as 
\begin{equation}
c^*(v) := \sup_{w \in \real^{q}} \bigl(\langle v, w\rangle - c(w) \bigr),\quad v\in\bR^q. \label{eq:cvxconjdef}
\end{equation}
For a convex function $c$, the {\em subdifferential set $\partial c$} of $c$ at each point $v\in\bR^q$ is defined as 
\[ 
\partial c(v) := \bigl\{ w \in \real^q : \text{for any }v' \in \real^q,\  c(v') - c(v) \geq \langle v'- v, w \rangle \bigr\}, 
\]
and each element in $\partial c(v)$ is called a {\em subgradient vector} of $c$ at $v$.

\begin{lemm}[See, e.g., {\cite[D.1 and E.1]{FundConvAnalysis2001HiriartUrrutyLemarechal}}]\label{lemm:cvxfcnprop}  
Let $c: \real^q \to (-\infty,+\infty]$ be a function such that $c\not \equiv +\infty$. 
\begin{enumerate}[(i)]
 \item \label{item:usualdiff} If $c$ is differentiable at a point $v \in \real^q$, then we have $\partial c(v)=\{\grad c(v)\}$, where $\grad c(v)\in\bR^q$ is
the gradient vector of $c$ at $v$. 
 \item \label{item:doubleconj}  
The function $c^{**}=(c^*)^*$ is characterized as the function such that for every $u \in \real$,
\begin{gather}
 \conv\bigl(L_{c,\le u}\bigr)=L_{c^{**},\le u}. 
\end{gather}
In particular, $c^{**}=c$ if $c$ is convex.
 \item \label{item:minsubdiff} 
If $c$ is convex, then $c^*(v) \geq -c(0)$ for any $v \in \real^q$, and the equality holds if and only if $v \in \partial c(0)$.
\end{enumerate}
\end{lemm}

\subsection{Well-definedness of the order functions}
In the rest of this section, 
shall we come back to our situation in Section \ref{sec:intro}.
Recall that for $m\gg 1$, the order function $\ord\rho_x$ on $\BT_G(K)$
is associated to $\rho\in L_m$ and $x\in X_K(K)$, and is independent of the choice of evaluations  $\lshf^m|_x \simeq K$. 

For completeness, we include herewith a proof of the following. 
\begin{lemma}\label{th:const}
For any $m_1,m_2\gg 1$, 
any non-zero $\rho_1 \in \invsgen{m_1}, \rho_2 \in \invsgen{m_2}$, and any $x \in X_K(K)$,
if $\rho_i (x) \neq 0$ for $i = 1,2$, 
then the difference 
$\ord(\rho_1)_x/m_1-\ord(\rho_2)_x/m_2$ 
is constant on $\BT_G(K)$. 
\end{lemma}
\begin{proof}
Noting that $\ord ((\rho_i)^{M})_x = M\cdot\ord(\rho_i)_x$ for any $M\in\bN$, we assume without loss of generality that $m_1 = m_2=:m$ and $\cL^m$ is very ample. 
By \eqref{eq:invord}, 
for any non-zero $\rho\in \invs$ and any $[g] \in \BT_G(K)$, we have
\begin{align}
 (\ord \rho_x)([g]) - (\ord \rho_x)([e])
    & = \min_{s \in L_m} \ord\bigl(s(x)\bigr) - \min_{s \in L_m} \ord \bigl((\sigma_* (g) \cdot s ) (x)\bigr), \label{eq:difference}
\end{align}
which is independent of $\rho$, so that
$\ord(\rho_1)_x-\ord(\rho_2)_x$ is constant on $\BT_G(K)$.
\end{proof}

\begin{rem}\label{th:eval}
In computing the function $\ord \rho_x$ on $\BT_G(K)$, it is {\em sometimes} 
convenient to fix and use an evaluation $\lshf^m|_x \simeq K$ normalized as 
\begin{gather}
    \min_{s \in L_m} \ord\bigl(s(x)\bigr) = 0, \label{eq:normalize}
\end{gather}
and then we also use the induced evaluation $\lshf_k^m|_{\spe x} \simeq k$ by base change, so that
for any section $s \in L_m = H^0(X,\lshf^m )$,
letting $s_k \in H^0(X_k, \lshf^m_k)$ be the base change of $s$, we have
\begin{gather}
 s_k(\spe x) = \bigl(s(x)\text{ mod }\mathfrak{m}_K\bigr). \label{eq:normalizered}
\end{gather}
\end{rem}

\subsection{Weight polytopes, auxiliary functions, and differences}
Now pick also $m\gg 1$, $x\in X_K(K)$, 
and $\rho\in \invs$ such that $\rho (x) \neq 0$, and also pick a maximal torus $\Lambda:\bG_m^r\to G$.
Fix an evaluation $\lshf^m|_x \simeq K$ normalized as \eqref{eq:normalize} and a weight decomposition $(b_i)_{i=1}^{N_{\cL^m}}$ of $L_m$ having the weights $(w_i)_{i=1}^{N_{\cL^m}}$ (see \eqref{eq:weitdecomp}). Set
\[ 
\obody := \conv \Bigl( \bigl\{ ( w_i , y ) \in \bR^r\times\bR=\real^{r+1} : y \geq \ord\bigl(b_i(x)\bigr)(\ge 0\text{ under \eqref{eq:normalize}}) \bigr\}\Bigr),  
\]
and we introduce two auxiliary functions $\bdobody : \real^r \to [0,+ \infty]$ and $W_{\Lambda,x}: \real^r \to [0,+ \infty]$ as
\begin{align*}
 \bdobody (w) &:= \min_{(w,y) \in \obody}y,\quad w\in\bR^r,\quad\text{and}\\
 W_{\Lambda, x} (w) &:=\min_{i:\,w_i =w}\ord\bigl(b_i(x)\bigr),\quad  w\in\bR^r,
\end{align*}
where $\min_\emptyset = +\infty$ by convention; then $\bdobody = (W_{\Lambda , x})^{**}$, so in particular the function $\bdobody$ is convex and $(\bdobody)^* = (W_{\Lambda , x})^*$.

The weight polytopes introduced in Subsection \ref{sec:wp}
and the auxiliary function $\bdobody$ are related as follows.

\begin{lemma}\label{lemm:wtpoly} 
We have the equalities 
\begin{gather} 
\Delta ( x , \Lambda_K )
=(\bdobody)^{-1}\bigl([0,+ \infty)\bigr)  
\quad\text{and}\quad
\Delta (\spe x , \Lambda_k )
=(\bdobody)^{-1}(0).\label{eq:kbody}
\end{gather}
\end{lemma}

\begin{proof} 
Noting that
$\{ w_i : b_i(x) \neq 0 \}=\{ w_i : \ord(b_i(x)) < +\infty \}$ and that the convex hull of the right hand side equals that of $W_{\Lambda, x}^{-1}([0,+\infty))$, we have
\begin{align*}
\Delta (x, \Lambda_K ) 
= \conv\bigcup_{u \ge 0}L_{W_{\Lambda, x},\le u} 
&= \bigcup_{u \ge 0}\conv\bigl(L_{W_{\Lambda, x},\le u}\bigr)\\ 
&=\bigcup_{u\ge 0}L_{W_{\Lambda, x}^{**},\le u} 
=\bigcup_{u\ge 0}L_{\bdobody,\le u} =(\bdobody)^{-1}\bigl([0,+ \infty)\bigr) 
\end{align*}
by Lemma \ref{lemm:cvxfcnprop}\eqref{item:doubleconj}.
Similarly, noting that 
the fixed weight decomposition $(b_i)_{i = 1}^{N_{\lshf^m}}$ of $L_m = H^0(X,\lshf^m )$ induces by base change a weight decomposition $(b_{i,k})_{i = 1}^{N_{\cL^m}}$ of $H^0(X_k,\lshf^m_k)$ still having the weights $(w_i)_{i=1}^{N_{\cL^m}}$ and that $\{ w_i : b_{i,k}(\spe x) \neq 0 \}\overset{\eqref{eq:normalizered}}{=}\{ w_i :  \ord b_i(x) = 0  \}$, we also have $\Delta (\spe x, \Lambda_K )=(\bdobody)^{-1}(0)$.
\end{proof}  
 
Let $\delta_x:\BT_G(K)\to\bR$ be the difference function
\begin{gather}
\delta_x ([g]) := \ord \rho_x([g]) -  \ord \rho_x([e]) \label{eq:defdelta}
\end{gather}
of the function $\ord\rho_x$ on $\BT_G(K)$. In the rest of this subsection, assume that the maximal torus $\Lambda$ is chosen as in Theorem \ref{rbuild}(\ref{aparttor}). On the apartment $A=f_{\Lambda,e}(\bR^r)$ in $\BT_G^\real (K)$, the functions $\delta_x$ and $\bdobody$ are related as in the following proof.
On the apartmentwise convex extendability of $\delta_x$, see also \cite{SternWewers26} in the discrete valued field case. 

\begin{lemma} \label{lemm:dap}
Let $f : \real^r \to \BT_G^\real (K)$ be an apartment and set $A := f(\real^{r})$. Then
the restriction $\delta_x|(A \cap \BT_G(K))$ 
extends uniquely to a convex function $\delta_A$ on $A$, which attains its infimum on $A\cap(f(\Gamma^r))$. 
\end{lemma}

\begin{proof}
By Theorem \ref{rbuild}\eqref{aparttor}, 
we have $f=f_{\Lambda,g}$
for some  
$g \in G(K)$, and then for every $t=(t_1,\ldots,t_r)\in(K^\times)^r$, setting $v_t:=(\ord t_1,\ldots,\ord t_r)\in\Gamma^r$, 
\begin{gather}
 \delta_x\bigl(f(v_t)\bigr)=\delta_x([\Lambda(t)\cdot g])=\delta_{g\cdot x}([\Lambda(t)])=\delta_{g\cdot x}\bigl(f_{\Lambda,e}(v_t)\bigr), \label{eq:basepoint}
\end{gather}
recalling that $\ord \rho_{g \cdot x} (h) = \ord \rho_x(h \cdot g)$ for every $h\in G(K)$. Hence we assume without loss of generality that $g=e$ replacing $x$ with $g\cdot x$ if necessary.

Under this assumption $g=e$ so $f=f_{\Lambda,e}$, for every $t\in(K^\times)^r$, we compute
\begin{align*}
\delta_x\bigl(f(v_t)\bigr)
    & = \ord \rho_x ([\Lambda (t)]) - \ord \rho_x ([e]) \notag \overset{\eqref{eq:invord}}{=}
- \min_{s \in L_m} \ord\bigl(((\sigma_*(\Lambda(t))) \cdot s) (x)\bigr) + \min_{s\in L_m}\ord\bigl(s(x)\bigr) \\
    &\overset{\eqref{eq:normalize}}{=}
- \min_{i} \ord\bigl(((\sigma_*(\Lambda(t))\cdot b_i) (x)\bigr) +0 \overset{\eqref{eq:weitdecomp}}{=}
-\min_i\bigl( \bigl\langle w_i,v_t \bigr\rangle +\ord\bigl(b_i(x)\bigr) \bigr) \notag \\
    & =-\min_i\bigl( \bigl\langle w_i,v_t \bigr\rangle + W_{\Lambda, x} (w_i)\bigr)
    = - \inf_{w \in \real^r} \bigl( \langle w , v_t \rangle + W_{\Lambda, x} (w) \bigr)\\
    &= \sup_{w \in \real^r} \bigl( \langle w, -v_t \rangle - W_{\Lambda,x}(w)\bigr) 
    \overset{\eqref{eq:cvxconjdef}}{=} (W_{\Lambda,x})^*(-v_t ) = (\bdobody )^*( -v_t ), 
\end{align*}
which yields the unique convex extension $\delta_A$
of the restriction $\delta_x|(A \cap \BT_G(K))$
so that 
\begin{gather}
 (\delta_A \circ f)(v) = (\bdobody )^*(-v),\quad v\in\bR^r.\label{eq:pullback}
\end{gather}
Then by Lemma \ref{lemm:cvxfcnprop}(\ref{item:minsubdiff}) applied to the convex function $\bdobody$, we have both
\begin{gather}
\min_{\bR^r}(\delta_A\circ f)=-\bdobody (0)\quad\text{and}\label{eq:minimum}\\
 (\text{the minimum locus of the pullback }\delta_A \circ f)=-\bigl(\partial ( \bdobody )\bigr)(0)\subset \bR^r,\label{eq:minlocus}
\end{gather}
and indeed $(-(\partial ( \bdobody ))(0))\cap\Gamma^r\neq\emptyset$ since $w_i \in \zahl^{r}$ for every $i$ and $\Gamma$ is a $\rat$-vector space.
\end{proof}

The following is a consequence of Definition \ref{def:rbuild}(\ref{apcommon}) and Lemma \ref{lemm:dap}.
\begin{lemma}\label{lemm:delta}
The difference function $\delta_x$ on $\BT_G(K)$
extends uniquely to a convex function $\hat{\delta}_x$ on the Euclidean building $\BT^\real_G(K)$. 
\end{lemma}

\section{Proof of Theorem \ref{th:genmininvloc}}\label{sec:coincide}
Pick $m\gg 1$, $x\in X_K(K)$, and $\rho \in I_m$ such that $\rho (x) \neq 0$, and fix an evaluation $\lshf^m|_x \simeq K$ normalized as \eqref{eq:normalize}.
We note that by the definition $\delta_x = \ord \rho_x - \ord \rho_x ([e])$, $\MIL_x$ and the minimal locus of $\delta_x$ are identical. 

Suppose that $[g]\not \in \RSS_x$, and we assume without loss of generality that $[g]=[e]$ replacing $x$ with $g\cdot x$ if necessary (see \eqref{eq:basepoint}). 
Then by the implication ``(\ref{item:nonnull})$\Rightarrow$(\ref{item:ss})'' in Theorem \ref{th:hmcrit} applied to $\spe x$, there is some maximal torus $\ol{\Lambda} : 
(\bG_m^r)_k \to G_k$, which lifts to a maximal torus $\Lambda : \bG_m^r\to G$ (i.e.\ $\Lambda_k=\overline{\Lambda}$) by Lemma \ref{lemm:toruslift}, such that
fixing a weight decomposition \eqref{eq:weitdecomp} of $L_m$, 
\[ 
0 \not \in \Delta \bigl( \spe x , \ol{\Lambda }\bigr)
\overset{\eqref{eq:kbody}}{=}(\bdobody)^{-1}(0),
\]
which together with \eqref{eq:minimum}
yields $\min_{f_{\Lambda,e}(\Gamma^r)}\delta_x=-\bdobody (0) < 0=\delta_x ([e])$.
Hence we have $([g]=)[e]\not\in\MIL_x$.
 
Suppose alternatively that $[g] \in \RSS_x$.
Let us fix a maximal torus $\Lambda:\bG_m^r\to G$ as in Theorem \ref{rbuild}(\ref{aparttor}). Then 
for every $[h]\in\BT_{G}(K)\setminus\{[g]\}$, 
recalling Definition \ref{def:rbuild}(\ref{item:join}), 
we have $\{[g],[h]\}\subset f(\Gamma^r)$ for some apartment $f:\bR^r\to\BT_G^{\bR}(K)$, and we assume without loss of generality that $[g]=[e]$ and $f=f_{\Lambda,e}$ replacing $\Lambda$ and $x$ with some $G$-conjugate of $\Lambda$ and some point in $G{(K)}\cdot x$
respectively if necessary (see \eqref{eq:basepoint}).
By the implication ``(\ref{item:ss}) $\Rightarrow$ (\ref{item:nonnull})'' in Theorem \ref{th:hmcrit} applied to $\spe x$ and $\Lambda_k$, fixing a weight decomposition \eqref{eq:weitdecomp} of $L_m$, we have
\[ 
0 \in \Delta \bigl( \spe x , \Lambda_k\bigr)
\overset{\eqref{eq:kbody}}{=}(\bdobody)^{-1}(0), 
\]
which together with \eqref{eq:minimum}
yields $\delta_x([h])\ge\min_{f_{\Lambda,e}(\Gamma^r)}\delta_x=-\bdobody (0)=0=\delta_x ([e])$.
Hence we also have $([g]=)[e]\in\MIL_x$. 
If moreover $([g] = ) [e] \in \RS_x$, then we even have 
$0 \in \Int\Delta (\spe x,  \Lambda_k )=\Int((\bdobody)^{-1}(0))$,
which together with
Lemma \ref{lemm:cvxfcnprop}(\ref{item:usualdiff}) and \eqref{eq:minlocus} yields
$(\text{the minimum locus of }\delta_x\circ f_{\Lambda,e}) 
= -(\partial (\bdobody ))(0)=\{0\}$
so $\delta_x([h])>\delta_x(f_{\Lambda,e}(0))=\delta_x([e])$.
Hence we even have $\MIL_x=\{[e]\}(=\{[g]\})$ in this stable reduction case.
\qed

\section{Proof of Theorem \ref{th:existence}: non-emptiness of $\MIL_x$ in the complete case}\label{sec:SphComp}

We prepare an auxiliary lemma. 
\begin{lemma}\label{lem:complation}
    Let $\oring_K$ be a complete valuation ring and $K$ its fraction field. Let $M,N\in\bZ$, $M\leq N$, $\pi :K^N=K^M\times K^{N-M}\to K^M\cong K^N/K^{N-M}$ the projection, and $Z$ a Zariski closed subset of $K^N$. If there is a sequence $(z_n)$ in $Z\cap(\oring_K)^N$ such that $\lim_{n\to\infty}\pi (z_n)=0$ in the metric topology, then $Z \cap (\pi^{-1}(0)) 
\neq \emptyset$.
\end{lemma}

\begin{proof}
    We assume without loss of generality that $Z$ is defined as the common zeros of polynomials having $\oring_K$-coefficients. For each $n\in\bN$, set the ideal $J_n := \{a \in K : \ord a \geq n\}$ in $\cO_K$ and let $\pi_n:(\oring_K)^N\to (\oring_K/J_n)^N$ be the projection. Then under the assumption, we have
    \begin{gather*}
Z_n:=\pi_n(Z)\cap \bigl( \{ 0_{\cO_K} \text{ mod } J_n \}^M \times (\oring_K/J_n)^{N-M}\bigr)\neq \emptyset,
    \end{gather*}
    and in turn $Z\cap\pi^{-1}(0)=\bigcap_n \pi_n^{-1} (Z_n )\neq\emptyset$
by the completeness of $\oring_K$, which is $\simeq \varprojlim \oring_K/J_n$.
\end{proof}

Let us also denote by $\|\cdot\|$ the Euclidean norm on $\bR^r$, and by $\bS^r$ the unit sphere in $\bR^r$.

\begin{lemma}\label{lemm:sublevelbdd}
 Pick $x\in X^{\st}(\cL)$. For any $u \in \real$, the sublevel set 
    $\hat{L}_{\hat{\delta}_x , \leq u}$ 
 is bounded in the (not necessarily complete) $\operatorname{CAT}(0)$-space $\BT_G^\real (K)$.
\end{lemma}

\begin{proof}
Pick $x\in X^{\st}(\cL) \subset X_K(K)$.
Suppose to the contrary that for some $u\in\bR$, there is a sequence 
$(a_n)$ in $\BT_G^\real (K)\cap L_{\hat{\delta}_x , \leq u} $ such that 
$\lim_{n\to\infty}d\bigl(a_n,[e]\bigr)=+\infty$. Then increasing $u$ and replacing $x$ with some point in $G(K)\cdot x$
if necessary, we assume that $u=0$ and that $a_n\in\BT_G(K)$ for every $n\in\bN$.

Let us fix a maximal torus $\Lambda : \gmult^r \to G$ as in Theorem \ref{rbuild}(\ref{aparttor}), so that
for every $n \in \bN$, recalling also Definition \ref{def:rbuild}(\ref{item:join}), there is an apartment $f_n = f_{\Lambda, \gamma_n} : \real^r \to \BT_G^\real(K)$ for some $\gamma_n\in G(K)$ such that $\{ a_n,[e]\}\subset f_n(\Gamma^r)$;
we assume without loss of generality that
$f_n^{-1}([e]) = 0$ or equivalently that
\begin{gather}
  [\gamma_n]=[e] \label{eq:trivial}
\end{gather}
(by replacing $[\gamma_n]$ with some point in $f_n(\Gamma^r)$ if necessary), and then
choose $g_n \in  \bigl( \Lambda(\bG_m^r) \bigr) (K) \cdot \gamma_n \subset G(K)$ so that
\begin{gather*}
 [g_n]=a_n. 
\end{gather*}

For each $n\in\bN$, set 
\begin{align*}
    v_n &:= f_n^{-1}(a_n)\in\Gamma^r,\quad
    \text{so that }\lim_{n\to\infty}\|v_n\|=+\infty,\quad\text{and} \\
 u_n &:= v_n/\|v_n\|\in\bS^n,
\end{align*}
and then by the completeness of $\bS^r$, there is
\begin{gather*}
    u_0:=\lim_{n\to\infty}u_n\in\bS^n
\end{gather*}
taking a subsequence of $(a_n)$ if necessary.  

Pick $m\gg 1$ (independently of $x$)
and fix an evaluation $\lshf^m|_x \simeq K$
normalized as \eqref{eq:normalize} and a weight decomposition $(b_i)_{i=1}^{N_{\cL^m}}$ of $L_m$ having the weights $(w_i)_{i=1}^{N_{\cL^m}}$ (see \eqref{eq:weitdecomp}), so that 
\begin{gather*}
    \min_i\ord\bigl(b_i(x)\bigr)=\min_{s\in L_m}\ord\bigl(s(x)\bigr)=0.
\end{gather*} 
For every $n\in\bN$, we have
\begin{gather}
 \min_i\ord\bigl(b_i(g_n\cdot x)\bigr)=\min_{s \in L_m}\ord\bigl(s(g_n \cdot x)\bigr) \overset{\eqref{eq:difference} \& \eqref{eq:defdelta}}{=} -\delta_x([g_n]) =-\hat{\delta}_x(a_n)\ge 0, \label{eq:positive}
\end{gather}
and similarly,
choosing $h_n \in \bigl( \Lambda( \bG_m^r )\bigr) (K)  \cdot \gamma_n\subset G(K)$ so that 
\begin{gather}
[h_n]=f_n\Bigl(\frac{v_n}{2}\Bigr)\Bigl(=f_n\Bigl(v_n-\frac{v_n}{2}\Bigr)\Bigr)\label{eq:middle}
\end{gather}
and using the convexity of the function $\hat{\delta}_x$, we also have
\begin{align}
    \min_i \ord \bigl(b_i(h_n \cdot x)\bigr) 
    = - \hat{\delta}_{x}([h_n])
    \geq  - \frac{\hat{\delta}_{x}([e]) + \hat{\delta}_{x}(a_n) }{2} \geq -\frac{0+0}{2}=0 \label{eq:hnxinlattice}
\end{align}
(so $b_i(h_n \cdot x) \in \cO_K$). Then
noting that for every $i$,
$\lim_{n\to\infty}\langle u_n,w_i\rangle=\langle u_0,w_i\rangle$, we also have
\begin{align}
     \min_{i:\,\langle u_0, w_i\rangle > 0 } \ord\bigl(b_i(h_n\cdot x)\bigr)
    & \overset{\eqref{eq:weitdecomp}}{=} \min_{i :\,\langle u_0, w_i \rangle> 0}
\Bigl(\Bigl\langle\frac{v_n}{2}, w_i\Bigr\rangle
+ \ord\bigl(b_i(\gamma_n \cdot x)\bigr)\Bigr)\notag\\
    & \overset{\eqref{eq:trivial}}{\ge }\min_{i :\,\langle u_0, w_i \rangle> 0}\Bigl\langle\frac{v_n}{2}, w_i\Bigr\rangle+\min_i\ord\bigl(b_i(x)\bigr)\notag\\
    & \geq \frac{\|v_n\|}{2}\cdot
\min_{i :\, \langle u_0, w_i \rangle >0}\langle u_n,w_i\rangle + 0
\to +\infty   \label{eq:wtonhalf}
\end{align}
and 
\begin{align}
 \min_{i:\, \langle u_0, w_i \rangle < 0}  \ord\bigl( b_i(h_n\cdot x)\bigr)
&\overset{\eqref{eq:middle}\&\eqref{eq:weitdecomp}}{=}
\min_{i :\,\langle u_0, w_i \rangle< 0}
\Bigl(\Bigl\langle-\frac{v_n}{2}, w_i\Bigr\rangle
+ \ord\bigl(b_i(g_n \cdot x)\bigr)\Bigr)\notag\\
& \geq 
\min_{i :\,\langle u_0, w_i \rangle< 0}\Bigl\langle-\frac{v_n}{2}, w_i\Bigr\rangle+ \min_i \ord \bigl(b_i(g_n\cdot x)\bigr)\notag\\
&\overset{\eqref{eq:positive}}{\ge}\frac{\|v_n\|}{2}\cdot \min_{i :\, \langle u_0, w_i \rangle < 0} 
(-\langle u_n, w_i\rangle) 
+0\to+\infty\label{eq:wtonhalfneg}
\end{align}
as $n\to\infty$.

Let us write the ambient $K$-linear space $M:= (L_m)_K^{\vee}$ of the affine cone of $X$, which is spanned by 
the (free) basis $(b_i^*)_{i=1}^{N_{\cL^m}}$ of the $\cO_K$-module $(L_m)^\vee$ (which is an $\cO_K$-lattice of $M$ canonically) dual to $(b_i)$. Then under the assumption that $x\in X^{\st}(\cL)$, by Remark \ref{rem:actionlift}, the orbit 
$Z := \sigma_*^\vee (G) \cdot \widetilde{x}$
of the point $\widetilde{x} := \sum_{i} (b_i(x)) \cdot b_i^*\in (L_m)^\vee$
under the dual action of $G$ is a Zariski closed subset in $M$. Let
$M^{u_0}$ be the subspace of $M$ spanned by $\{b_i^*:\langle w_i,u_0\rangle = 0\}$,
and let $\pi : M \to M/(M^{u_0})$ the projection.

By \eqref{eq:hnxinlattice}, we still have $(\sigma_*^\vee (h_n)) \cdot \widetilde{x} = \sum_{i}(b_i( h_n \cdot x )) \cdot b_i^* \in (L_m)^\vee$ for every $n\in\bN$, and by \eqref{eq:wtonhalf} and \eqref{eq:wtonhalfneg},
we also have $\lim_{n\to\infty}\pi (\sigma_*^\vee (h_n) \cdot \widetilde{x}) =0$ in the metric topology.
Hence using Lemma \ref{lem:complation}, we have
\begin{gather*}
     \bigl(\sigma_*^\vee (g)\bigr)\cdot\tilde{x}
    =\sum_{i} \bigl(b_i( g \cdot x )\bigr) \cdot b_i^* \in Z\cap(M^{u_0})
\end{gather*}
for some $g\in G(K)$,
and then we still have $g\cdot x \in X^{\st}(\cL)$ recalling Remark \ref{rem:actionlift} and also have
\begin{gather*}
b_i (g\cdot x) = 0 \text{ if }\langle w_i , u_0 \rangle \neq 0. 
\end{gather*} 
For such $g$, by the implication  ``(\ref{item:ss}) $\Rightarrow$ (\ref{item:nonnull})'' in Theorem \ref{th:hmcrit} and $\|u_0\|=1\neq 0$,
we must have not only $0 \in\Int\Delta (\Lambda , g \cdot x)$ but also
\[
\Delta (\Lambda , g\cdot x)\overset{\eqref{eq:wpdecomp}}{=}
\conv \bigl(\{ w_i : b_i (g \cdot x ) \neq 0  \}\bigr) \subset \conv \bigl( \{ w_i : \langle w_i , u_0 \rangle = 0 \} \bigr) \subset (\bR u_0)^\perp\subset\bR^r 
\]
so $\Int\Delta (\Lambda , g\cdot x)=\emptyset$,
where $\perp$ means the orthonormal complement. This is a contradiction.
\end{proof}

\begin{proof}[Proof of Theorem \ref{th:existence}]
Under the assumption in Theorem \ref{th:existence},
the Euclidean building $\BT_G^\real (K)$ is a {\em complete} 
metric $\operatorname{CAT}(0)$-space. 
Pick $x\in X^{\st}(\cL)$. Then by Lemma \ref{lemm:sublevelbdd} above
and Theorem \ref{th:convexmin} below, 
$\hat{\delta}_x$ attains its infimum at some point $a_0\in\BT_G^{\bR}(K)$.

Recalling Definition \ref{def:rbuild}(\ref{item:join}), we have 
$\{a_0,[e]\}\subset A_0$ for some apartment $A_0$ of $\BT^\real_G(K)$. Then by Lemma \ref{lemm:dap}, there is $[g_1] \in A_0 \cap(\BT_G(K))$ such that
$\delta_x( [g_1] )= \delta_{A_0}([g_1]) = \min_{a \in A_0} \delta_{A_0}(a) \le \hat{\delta}_x(a_0)$ so,
recalling that 
$\MIL_x$ and the minimal locus of $\delta_x$ are identical, the function
$\ord\rho_x$ on $\BT_G(K)$ attains its infimum at $[g_1]$.
\end{proof} 

In the above proof, we used the following general theorem.

\begin{thm}[see, e.g., {\cite{Bacak2014ConvexAnalyAndOptimInHadamSp}}]\label{th:convexmin}
Let $X$ be a complete $\operatorname{CAT}(0)$-space and $c : X \to\bR$ a convex function on $X$. 
Then the minimum of $c$ exists if the sublevel set
$L_{c,\leq u}$ is bounded in $X$ for every $u\in\bR$. 
\end{thm}

\section{Applications and computations}\label{sec:apply}
We give applications of Theorem \ref{th:genmininvloc} and computation of $\MIL_\varphi$ so of $\RSS_\varphi$ in the dynamical setting mentioned in Section \ref{sec:intro}. Recall that $X=X_{n,d}= \overline{\Hom_d(\mathbb{P}^n)} \simeq \proj^{N-1}$, where
\begin{gather}
 N=N_{n,d}= (n+1) \binom{n+d}{d},\label{eq:dim} 
\end{gather}
equipped with $\PGL_{n+1}$-$,   
\GL_{n+1}$-, or $\SL_{n+1}$-linearized line bundle $\cL = \oshf_{\proj^{N-1}}(1)$. By a technical reason, we consider herewith the $\SL_{n+1}$-conjugation action.

Recall that, as a projective space of $\sln{n+1}$-representation, 
\[ 
X=\ol{\Hom_d(\mathbb{P}^n)} \simeq \proj (V_1^* \otimes V_d), 
\]
where $V_1 = K^{n+1}$ is the standard representation of $\sln{n+1}$ and $V_\nu := S_\nu V_1$ for the Schur functor $S_\nu$ associated with a partition $\nu$ (see \cite{fulton1997young}).
From now on, we assume that $\chara k = 0$. By Pieli's rule, there is an isomorphism of $\sln{n+1}$-representations 
\[ 
\iota : V_1^* \otimes V_d \to V_{(d+1,1^{n-1})} \oplus V_{d-1}, 
\]
which is taken over $\rat (\subset \oring_K)$.

Now for example set $(n,d) = (2,4)$, and pick $\alpha \in K^\times$ and consider $\varphi_\alpha \in\Hom_4(\bP^2)$,
\[ 
\varphi_\alpha ([x_0 : x_1 : x_2] ) = [ x_0^4 + \alpha x_0^2x_1x_2 : x_1^4 + \alpha x_0x_1^2x_2 : x_2^4 + \alpha x_0x_1x_2^2 ]. 
\]
Then under the (projectivization of the) isomorphism $\iota$, we have 
\[ 
\iota (\varphi_\alpha) = [( v,  x_0^3 + x_1^3 + x_2^3 + 3\alpha x_0x_1x_2 )] \in \proj (V_{(4,1)} \oplus V_{3})
\]
for some $v \in V_{4,1}$.
Now we recall Mumford's criterion of stability (resp. semistability) for a plane curve  \cite[Section 4.2]{git1994}: {\em a point $[f] \in  \proj (H^0( \proj^2 , \oring (3)) )$ is stable (resp. semistable) if the curve $C = V_+(f) ( = (f=0))$ in $\proj^2$ 
has no singular point (resp. cusp or triple point)}. In this case, the same stability conditions can be applied for $(v,f) \in V_{4,1} \oplus V_{3}$.

If $\ord \alpha \geq 0$ and $\alpha_k ^3 \neq -1$, then
\[ 
\iota\bigl( \spe (\varphi_\alpha)\bigr) = [ (v_k, x_0^3 + x_1^3 + x_2^3 +3 \alpha_k x_0x_1x_2)], 
\]
so that $\spe ( \varphi_\alpha ) \in X^{\st}_k(\cL_k)$ by the above criterion.
Then by the final statement in Theorem \ref{th:genmininvloc}, we have 
$\RSS_{\varphi_\alpha}= \RS_{\varphi_\alpha}=\{ [e] \}$.

On the other hand, if $\ord \alpha < 0$, then
\[ 
\iota \bigl(\spe (\varphi_\alpha )\bigr) = [(0, x_0x_1x_2 )], 
\]
so that $\spe ( \varphi_\alpha ) \in X^{\sst}_k(\cL_k)$ by the above criterion. Under the group scheme homomorphism 
\begin{gather*}
    \Lambda : \bG_m^2 \ni(t_1, t_2) \mapsto [ \diag (t_0 := (t_1t_2)^{-1}, t_1, t_2 ) ]\in \SL_3, 
\end{gather*}
we have
\[ 
\bigl( \Lambda (t_0,t_1,t_2) \cdot \varphi_\alpha\bigr)([x_0 : x_1:x_2]) = [ t_0^3 x_0^4 + \alpha x_0^2x_1x_2 : t_1^3 x_1^4 + \alpha x_0x_1^2x_2 : t_2^3 x_2^4 + \alpha x_0x_1x_2^2 ],
\]
so that choosing some $\rho\in \invs$ such that 
$\rho ( \varphi_\alpha ) \neq 0$ and recalling Theorem \ref{rbuild}(\ref{aparttor}), we compute
\begin{align*} 
    \RSS_{\varphi_\alpha} \cap \bigl(f_{\Lambda ,e} (\real^2)\bigr) 
    & \overset{\text{Thm  \ref{th:genmininvloc}}}{=} \MIL_{\varphi_\alpha} \cap \bigl(f_{\Lambda , e} (\real^2)\bigr) \\
& = 
\Bigl\{ [\Lambda (t_1,t_2)] \in \BT (\SL_3) : \ord \rho_{\varphi_\alpha}([\Lambda(t_1,t_2)]) 
= \min\rho_{\varphi_\alpha}
\bigl(\overset{\text{Thm  \ref{th:genmininvloc}}}{=} \rho_{\varphi_\alpha} ([e])\bigr)\Bigr\} \\
& = \Bigl\{ [\Lambda (t_1,t_2)] \in \BT (\SL_3) :\min_{i\in\{0,1,2\}} \ord(t_i^3) \geq \ord\alpha\Bigr\}\\ 
& = \Bigl\{ [\Lambda (t_1 , t_2)] \in \BT (\SL_3) : \min_{i\in\{0,1,2\}} ( \ord t_i ) \geq (\ord \alpha)/3  \Bigr\}  
\end{align*}
so $\RSS_{\varphi_\alpha}\neq\emptyset$.
For instance, choosing $(\ord t_1 , \ord t_2) = (1/3, 1/3) \cdot \ord \alpha$, we have
\begin{gather*}
   \iota \bigl( \spe  ( \Lambda (t_1,t_2) \cdot \varphi_\alpha ) \bigr) 
 = [( v'_k, ( (t_1^3 x_1^3 + t_2^3x_2^3) + 3 \alpha x_0x_1x_2  )_k ]\quad\text{for some }v' \in V_{4,1}, 
\end{gather*}
which is identified with (not an endomorphism but) a rational map
\[
\bigl( \spe (\Lambda (t_1,t_2) \cdot [\varphi_\alpha ] )\bigr):
[x_0,x_1,x_2]\mapsto[x_0x_1x_2: u_1 x_1^3+x_0x_1x_2 : u_2x_2^3 + x_0x_1x_2 ]
\]
of $\bP^2_k$ having the indeterminacy locus $\{ [1:0:0]\}$ for some $u_0,u_1\in k$. 

\begin{acknowledgement}
R.G was supported by JSPS Grant-in-Aid for JSPS Fellows 25KJ0090, and
Y.O was partially supported by JSPS Grant-in-Aid for Scientific Research (C), 23K03129.
 This research was partly done during the authors' visiting the Research Institute for Mathematical Sciences,
an International Joint Usage/Research Center located in Kyoto University.
\end{acknowledgement}


\providecommand{\bysame}{\leavevmode\hbox to3em{\hrulefill}\thinspace}
\providecommand{\MR}{\relax\ifhmode\unskip\space\fi MR }
\providecommand{\MRhref}[2]{%
  \href{http://www.ams.org/mathscinet-getitem?mr=#1}{#2}
}
\providecommand{\href}[2]{#2}

\end{document}